\documentclass[10pt]{elsarticlemod}
\usepackage{dcolumn}
\usepackage{bm}
\usepackage{verbatim}       

\usepackage{amsthm}
\usepackage{amssymb}
\usepackage{amscd}

\usepackage{bm}
\usepackage{amstext}
\usepackage{amsmath}
\usepackage{amsfonts}
\usepackage{units}
\usepackage{mathrsfs}

\newcommand{\dd}{{\rm d}}

\newcommand{\bd}{\begin{definition}}                
\newcommand{\ed}{\end{definition}}                  
\newcommand{\bc}{\begin{corollary}}                 
\newcommand{\ec}{\end{corollary}}                   
\newcommand{\bl}{\begin{lemma}}                     
\newcommand{\el}{\end{lemma}}                       
\newcommand{\bp}{\begin{proposition}}            
\newcommand{\ep}{\end{proposition}}                
\newcommand{\bere}{\begin{remark}}                  
\newcommand{\ere}{\end{remark}}                     

\newcommand{\bt}{\begin{theorem}}
\newcommand{\et}{\end{theorem}}

\newcommand{\be}{\begin{equation}}
\newcommand{\ee}{\end{equation}}

\newcommand{\bit}{\begin{itemize}}
\newcommand{\eit}{\end{itemize}}
\newtheorem{theorem}{Theorem}[section]
\newtheorem{corollary}[theorem]{Corollary}
\newtheorem{lemma}[theorem]{Lemma}
\newtheorem{proposition}[theorem]{Proposition}
\theoremstyle{definition}
\newtheorem{definition}[theorem]{Definition}
\theoremstyle{remark}
\newtheorem{remark}[theorem]{Remark}
\newtheorem{example}[theorem]{Example}

\def\hyph{-\penalty0\hskip0pt\relax}

\begin{document}

\begin{frontmatter}

\title{Convexity and quasi-uniformizability of closed preordered spaces}


\author{E. Minguzzi}

\address{ Dipartimento di Matematica e Informatica ``U.\ Dini'',
Universit\`a degli Studi di Firenze, Via S. Marta 3,  I-50139
Firenze, Italy.} \ead{ettore.minguzzi@unifi.it}



\begin{abstract}
\noindent In many applications it is important to establish if a
given topological preordered space has a topology and a preorder
which can be recovered from the set of continuous isotone functions.
Under antisymmetry this property, also known as
quasi-uniformizability, allows one to compactify the topological
space and to extend its order dynamics. In this work we study
locally compact $\sigma$-compact spaces endowed with a closed
preorder. They are known to be normally preordered, and it is proved
here that if they are locally convex, then they are convex, in the
sense that the upper and lower topologies generate the topology. As
a consequence, under local convexity they are quasi-uniformizable.
The problem of establishing local convexity under antisymmetry is
studied. It is proved that local convexity holds provided  the
convex hull of any compact set is compact. Furthermore, it is proved
that local convexity holds whenever the preorder is compactly
generated, a case which includes most examples of interest,
including preorders determined by cone structures over
differentiable manifolds. The work ends with some results on the
problem of quasi-pseudo-metrizability. As an application, it is
shown that every stably causal spacetime is quasi-uniformizable and
every globally hyperbolic spacetime is strictly
quasi-pseudo-metrizable.
\end{abstract}
%
%



\end{frontmatter}


\section{Introduction}
Topological preordered spaces are ubiquitous. They appear in the
study of dynamical systems \cite{bathia67}, general relativity
\cite{minguzzi09c}, microeconomics \cite{aumann62,bridges95},
thermodynamics \cite{lieb02} and computer science \cite{gierz03}. In
these applications it is important to establish if a topological
preordered space $(E,\mathscr{T}, \le)$ is quasi\hyph uniformizable,
namely, if  there is a quasi-uniformity $\mathcal{U}$ such that
$\mathscr{T}=\mathscr{T}(\mathcal{U}^{*})$ and $G(\le)=\bigcap
\mathcal{U}$. Taking into account a characterization of
quasi-uniformizability established by Nachbin \cite{nachbin65}, this
problem is equivalent to that of establishing if the topology and
the preorder of the space are determined by the family of continuous
isotone functions.

Hausdorff quasi\hyph uniformizable spaces are compactifiable
\cite{nachbin65,fletcher82} and, in general, the possibility of
restricting an analysis to the compact case brings several
simplifications. In other circumstances, the boundary (remainder)
involved in the compactification has special importance. For
instance, a good definition of spacetime boundary in general
relativity would allow us  to identify the singular spacetime points
\cite{geroch72}.


Quasi-uniformizable spaces are $T_2$-preordered spaces, thus $\le$
must be closed in order to have any chance to come from a
quasi-uniformity. In the various fields that in one way or the other
are connected with topological preordered spaces, it has been
discovered that it is indeed very convenient to study some new
closed preorder related to the original preorder. This is the
strategy of `prolongations' introduced by Auslander in dynamical
systems \cite{auslander64}, and rediscovered in a different setting
in relativity theory, where  Seifert \cite{seifert71,minguzzi07}
introduced a closed relation related to the causal relation, and
Sorkin and Woolgar  \cite{sorkin96} introduced the smallest closed
relation containing the causal relation (see Sect. \ref{nod}).


Given a $T_2$-preordered space $(E,\mathscr{T},\le)$ it is possible
to infer preorder normality provided $(E,\mathscr{T})$ is a
$k_\omega$-space \cite{minguzzi11f}. We recall that a topological
space is a $k_\omega$-space if there is a (admissible)  sequence of
compact sets $K_i$, $\bigcup_i K_i=E$, such that $O\subset E$ is
open if and only if $O\cap K_i$ is open in $K_i$. It is not
restrictive to assume $K_i\subset K_{i+1}$, and $K_1$ equal to any
chosen compact set. In this work we shall use the fact that locally
compact $\sigma$-compact spaces are $k_\omega$-spaces. Indeed, under
local compactness the properties: hemicompact $k$-space,
$k_\omega$-space, $\sigma$-compact,  and Lindel\"of  are equivalent
\cite{minguzzi11c}. Since we do not assume that $E$ is Hausdorff, we
remark that in our terminology a topological space is {\em locally
compact} if each point has a compact neighborhood. It is {\em
strongly locally compact} if at each point the neighborhood system
of the point has a base made of compact neighborhoods (not
necessarily closed).

Convex normally preordered spaces are quasi\hyph uniformizable
\cite[Prop.\ 4.7]{fletcher82} (i.e.\ they are completely regularly
preordered spaces \cite{nachbin65}), and quasi\hyph uniformizable
spaces  are convex closed preordered spaces. Unfortunately, although
a $T_2$-preordered $k_\omega$-space is normally preordered, preorder
normality does not imply quasi\hyph uniformizability as convexity is
missing.  Indeed, we shall give an example of a $T_2$-preordered
locally compact $\sigma$-compact space which is not convex (see
example \ref{bao}).


This work is devoted to the proof of the convexity and hence
quasi\hyph uniformizability of a large class of locally compact
$\sigma$-compact closed preordered spaces.

The main result of this work is the proof that  for these preordered
spaces local convexity and convexity are equivalent (Cor.\
\ref{bak}).
We then proceed to study local convexity, showing that it follows
from antisymmetry plus some other assumptions. We prove that local
convexity holds for {\em $k$-preserving spaces} (Theor.\ \ref{bdd}),
namely for those spaces for which  the convex hull of any compact
set is compact. The definition of $k$-preserving space is quite
important for the connection with global hyperbolicity in relativity
theory \cite{hawking73} (see Sect.\ \ref{nod}).

Furthermore, we show that if the order is compactly generated then
local convexity holds  (Cor.\ \ref{nja}). Joining this result with
the previous one  we infer that if, roughly speaking, both the
topology and order are generated `locally' then convexity holds
(Cor.\ \ref{bdr}). This case includes most examples of topological
preordered spaces of interest, including those in which the preorder
is induced by a distribution of tangent cones on a differentiable
manifold \cite{fathi12}. We shall compare our findings with similar
results obtained by Akin and Auslander in the study of dynamical
systems \cite{akin10b}.

Finally,  under second countability we obtain a result on the
quasi-pseudo-metrizability of the space which generalizes Urysohn's
theorem (Theor. \ref{ncu}), and under the $I$-space condition we are
able to assure the strict quasi-pseudo-metrizability of the space
(Theor. \ref{nci}). As an application, we prove that globally
hyperbolic spacetimes (see Sect.\ \ref{nod} for the definition) are
strictly quasi-pseudo-metrizable.

%

%
%

\subsection{Some reference results on mathematical relativity and causality
theory} \label{nod}
At places we shall illustrate our findings using
the topological ordered space given by the spacetime manifold
ordered with a  causal relation. Therefore,  it is worth
recalling some definitions and result from this field. The reader
can skip this section on first reading, returning to it whenever
this application is mentioned.

Let $M$ be a Hausdorff, connected, paracompact ($C^{r+1}$, $r\ge 0$)
manifold and let $g: M\to T^{*}M\otimes T^{*}M$ be a  ($C^r$, $r\ge
0$) Lorentzian metric, namely a pseudo-Riemannian metric of
signature $(-,+,\cdots,+)$. Non vanishing tangent vectors split into
spacelike, lightlike or timelike depending on the sign of $g(v,v)$,
$v\in TM$, respectively positive, null or negative. Lightlike or
timelike vectors are called causal. Assume that a continuous
timelike vector field can be defined over $M$, and call {\em
future} the cone of causal vectors including it. If this is not
possible there is always a double covering of $M$ with this
property, thus this is not a severe restriction. Once such a choice
of future cone has been made, the Lorentzian manifold is {\em time
oriented}. A {\em spacetime} $(M,g)$ is a time oriented Lorentzian
manifold. The simplest example of spacetime is the 2-dimensional
Minkowski spacetime, namely $\mathbb{R}^2$ with coordinates $(t,x)$,
metric $g=-\dd t^2+\dd x^2$ and time orientation given by the global
timelike vector $\partial_t$.

Let us observe that once a time orientation is given, any causal vector is
either {\em future directed} or {\em past directed} depending on
whether it belongs to the future cone. This terminology extends to
$C^1$  curves depending on the character of their tangent vector,
provided it is consistent throughout the curve.

The {\em causal relation} $J^+\subset M\times M$ over $M$ is defined
through: $(x,y)\in J^+$ if there is a future directed causal curve
from $x$ to $y$ or $x=y$. The {\em chronology relation} $I^+\subset
M\times M$ over $M$ is defined through: $(x,y)\in I^+$ if there is a
future directed timelike curve from $x$ to $y$. We have $J^+\circ
I^+=I^+\circ J^+=I^+$, and  $I^+$ is open in the product topology
\cite{minguzzi06c,hawking73}.
 Unfortunately, the causal relation is not
necessarily closed, as can be easily realized considering the
spacetime which is obtained removing a point  from the 2-dimensional
Minkowski spacetime.

The relation $K^+$ is by definition \cite{sorkin96} the smallest
closed and transitive relation containing $J^+$ and it exists
because $R:=M\times M$ provides an example of closed and transitive
relation containing $J^+$. Unfortunately, it is difficult to work
with $K^+$ since it is defined through its closure and transitivity
properties rather than through the more intuitive notion of causal
curve. Seifert \cite{seifert71} found another route to build a
closed and transitive relation. Let us write $g'>g$ if the timelike
cones of $g'$ contain the causal cones of $g$, and let $J^+_{g'}$ be
the causal relation for $(M,g')$. Seifert proved that
$J^+_S:=\bigcap_{g'>g} J^+_{g'}$ is indeed closed, transitive and contains
$J^+$.

A spacetime $(M,g)$ is said to be  {\em causal} if it does not
contain any closed causal curve. It is {\em stably causal} if there
is $g'>g$ such that $(M,g')$ is causal, namely if it is possible to
open the light cones everywhere over $M$ without introducing closed
causal curves. A relation $R$ is antisymmetric if $(x,y)\in R$ and
$(y,x)\in R$ implies $x=y$. It can be proved that the spacetime is
causal (resp.\ stably causal) iff $J^+$ (resp.\ $J^+_S$) is
antisymmetric \cite{minguzzi07}. It also turns out
\cite{minguzzi08b} that stable causality holds iff $K^+$ is
antisymmetric, and in this case $K^+=J^+_S$. Thus $J^+_S$ is really
the most natural closed and transitive relation that can be
introduced in a stably causal spacetime.

Let us write $J^+(x):=\{y: (x,y)\in J^+\}$ and $J^{-}(y):=\{x:
(x,y)\in  J^+ \}$, and if $X\subset M$, let $J^{\pm}(X):=
\bigcup_{x\in X} J^{\pm}(x)$. A spacetime is {\em causally
continuous} if the relation \[D^+=\{(x,y): y \in \overline{J^+(x)} \
\textrm{and} \ x \in \overline{J^-(y)} \},\] is antisymmetric (a
property known as {\em weak distinction}) and coincides with
$\overline{J^+}$ (a property known as {\em reflectivity}). It is not
hard to prove \cite{minguzzi07e} that $D^+$ is  transitive, thus
under causal continuity $\overline{J^+}$ is closed, transitive and
contains $J^+$. As a consequence, it is  the smallest relation with
such properties,  $K^+=\overline{J^+}$, and hence causal continuity
implies stable causality.

A spacetime is {\em causally simple} if it is causal and $J^+$ is
closed. Clearly, under causal simplicity $D^+=J^+$, thus causal
simplicity implies causal continuity (note that under causal
simplicity we have also $J^+=K^+=J^+_S$).

 Another important causality property is {\em
global hyperbolicity}. A spacetime $(M,g)$ is globally hyperbolic if
it is causal and for every compact set $K$, its convex causal hull
$J^+(K)\cap J^{-}(K)$ is compact. It can be shown that every
globally hyperbolic spacetime is causally simple \cite{minguzzi06c}.
These spacetimes are the most studied in mathematical relativity
because a spacetime is globally hyperbolic iff it admits a {\em
Cauchy hypersurface}, namely a topological hypersurface intersected
by any inextendible (i.e.\ with no endpoint) causal curve in exactly
one point \cite{hawking73}. Therefore, they are the spacetimes for
which the Cauchy problem of general relativity and that of wave
equations makes sense.


%

\subsection{Preliminaries on topological preordered spaces}

A topological preordered space  is a triple $(E,\mathscr{T},\le)$
where $(E,\mathscr{T})$ is a topological space and $\le$ is a {\em
preorder} on $E$, namely a reflexive and transitive relation. A
preorder is an {\em order} if it is antisymmetric (that is, $x\le y$
and $y\le x$ $\Rightarrow x=y$). For a topological preordered space
$(E,\mathscr{T},\le)$ our terminology follows Nachbin
\cite{nachbin65}.   With $i(x)=\{y: x\le y\}$ and $d(x)=\{y: y\le
x\}$ we denote the increasing and decreasing hulls, and we define
$[x]=d(x)\cap i(x)$.  The topological preordered space is {\em
$T_1$-preordered} (or {\em semiclosed preordered})  if $i(x)$ and
$d(x)$ are closed for every $x \in E$, and it is {\em
$T_2$-preordered} (or {\em closed preordered}) if the graph of the
preorder $G(\le)=\{(x,y): x\le y\}$ is closed.

Let $S\subset E$, we define $i(S)=\bigcup_{x\in S} \,i(x)$ and
analogously for $d(S)$. A subset $S\subset E$, is called {\em
increasing} if $i(S)=S$ and {\em decreasing} if $d(S)=S$.  It is
called {\em monotone} if it is increasing or decreasing. With $I(S)$
we denote the smallest closed increasing set containing $S$, and
with $D(S)$ we denote the smallest closed decreasing set containing
$S$. A subset $C$ is {\em convex} if it is the intersection of a
decreasing and an increasing set in which case $C=d(C)\cap i(C)$. A
subset $C$ is a $c${\em-set} \cite{kent85} if it is the intersection
of a closed decreasing and a closed increasing set in which case
$C=D(C)\cap I(C)$. The neighborhood of a point which is a $c$-set is
a $c$-neighborhood, and a $c$-set which is compact is a $c$-compact
set.  In the notation of this work the set inclusion $\subset$, is
reflexive, i.e.\ $X\subset X$.

A topological preordered space  is a {\em normally preordered space}
if it is $T_1$-preordered and for every closed decreasing set $A$
and closed increasing set $B$ which are disjoint, $A\cap
B=\emptyset$, it is possible to find an open decreasing set $U$ and
an open increasing set $V$ which separate them, namely $A\subset U$,
$B\subset V$, and $U\cap V=\emptyset$.

Given a reflexive relation $R$ on $E$,  a function $f: E\to
\mathbb{R}$ such that $(x,y) \in R \Rightarrow f(x)\le f(y)$ is an
{\em isotone} function. An  isotone function such that $(x,y) \in R
\textrm{ and } (y, x) \notin R \Rightarrow f(x)< f(y)$ is a {\em
utility} function.

In a normally preordered space, closed disjoint monotone sets as $A$
and $B$ above can be separated by a continuous isotone function
$f:E\to [0,1]$, that is $A\subset f^{-1}(0)$, $B\subset f^{-1}(1)$
(this is the preorder analog of Urysohn's separation lemma, see
\cite[Theor.\ 1]{nachbin65}). Normally preordered spaces are
$T_2$-preordered spaces, and $T_2$-preordered spaces are
$T_1$-preordered spaces.

A topological preordered space $E$ is {\em convex} at $x\in E$, if
for every open  neighborhood $O\ni x$, there are an open decreasing
set $U$ and an open increasing  set $V$ such that $x\in U\cap
V\subset O$ (this definition is  due to Nachbin \cite{nachbin48} and
is used in \cite{burgess77,kent85,kunzi04}, though the terminology
is not uniform in the literature). It is {\em locally convex} at
$x\in E$ if the set of convex neighborhoods of $x$ is a base for the
neighborhoods system of this point \cite{nachbin48,nachbin65}. It is
{\em weakly convex} at $x\in E$ if the set of convex open
neighborhoods of $x$ is a base for the neighborhoods system of this
point \cite{nachbin48,mccartan71}. The topological preordered space
$E$ is {\em convex} (locally convex, weakly convex) if it is convex
(resp.\ locally convex, weakly convex) at every point. Clearly,
convexity (at a point) implies weak convexity (at a point) which in
turn implies local convexity (at a point). Notice that according to
this terminology the statement ``the topological preordered space
$E$ is convex'' differs from the statement ``the subset $E$ is
convex'' (which is always true).

%

 A quasi-uniformity \cite{nachbin65,fletcher82}
is a pair $(X,\mathcal{U})$ such that $\mathcal{U}$ is a filter on
$X\times X$, whose elements contain the diagonal $\Delta$, and such
that if $V\in \mathcal{U}$ then there is $W\in \mathcal{U}$, such
that $W\circ W\subset V$. A quasi-uniformity is a uniformity if
$V\in \mathcal{U}$ implies $V^{-1} \in \mathcal{U}$, where
$V^{-1}=\{(x,y): (y,x)\in V\}$. To any quasi-uniformity
$\mathcal{U}$ corresponds a dual quasi-uniformity
$\mathcal{U}^{-1}=\{U: U^{-1}\in \mathcal{U}\}$.

From a quasi-uniformity $\mathcal{U}$ it is possible to construct a
topology $\mathscr{T}(\mathcal{U})$ in such a way that a base for
the filter of neighborhoods at $x$ is given by the sets of the form
$U(x)$ where $U(x)=\{y: (x,y)\in U\}$ with $U \in \mathcal{U}$. In
other words, $O\in \mathscr{T}(\mathcal{U})$ if for every $x\in O$
there is $U\in \mathcal{U}$ such that $ U(x) \subset O$.

Given a quasi-uniformity $\mathcal{U}$, the family $\mathcal{U}^*$
given by the sets of the form $V\cap W^{-1}$, $V,W\in\mathcal{U}$,
is the coarsest uniformity containing $\mathcal{U}$. The symmetric
topology of the quasi-uniformity is $\mathscr{T}(\mathcal{U}^{*})$.
Moreover, the intersection $\bigcap \mathcal{U}$ is the graph of a
preorder on $X$ (see \cite{nachbin65}), thus given a
quasi-uniformity one naturally obtains a topological preordered
space $(X,\mathscr{T}(\mathcal{U}^{*}),\bigcap \mathcal{U})$. The
topology $\mathscr{T}(\mathcal{U}^{*})$ is Hausdorff if and only if
the preorder $\bigcap \mathcal{U}$ is an order \cite{nachbin65}.

Nachbin proves \cite[Prop.\ 8]{nachbin65} that a topological
preordered space $(E,\mathscr{T},\le)$ comes from a quasi-uniformity
$\mathcal{U}$, in the sense that
$\mathscr{T}=\mathscr{T}(\mathcal{U}^{*})$ and $G(\le)=\bigcap
\mathcal{U}$, if and only if $E$ is a {\em completely regularly
preordered} space ($T_{3\nicefrac{1}{2}}$-preordered space,
Tychonoff-preordered space), namely if and only if the following two
conditions hold:
\begin{itemize}
\item[(i)] $\mathscr{T}$ coincides with the initial topology generated by the set of continuous
isotone functions $g:E\to [0,1]$,
\item[(ii)] $x\le
y$ if and only if for every continuous isotone function $f:E\to
[0,1]$, $f(x)\le f(y)$.
\end{itemize}
Completely regularly preordered spaces are convex $T_2$-preordered
spaces (convexity follows from (i) see \cite[Prop.\ 6,
Cap.II]{nachbin65}, and the closure of the preorder follows from
(ii)). Contrary to what happens in the usual discrete-preorder case,
normally preordered spaces need not be completely regularly
preordered spaces (see example \ref{bao}), nevertheless the preorder
analog of  Urysohn's separation lemma implies that convex normally
preordered spaces are completely regularly preordered spaces.
Completely regularly ordered spaces admit the Nachbin's
$T_2$-ordered compactification $nE$ (see \cite{fletcher82} and \cite{minguzzi11b} for the preorder case).

\subsection{Preliminary results on convexity}


A theorem by Nachbin states that every compact $T_2$-ordered space
is convex \cite[p.\ 48]{nachbin65}. Unfortunately, this theorem
assumes the compactness of the space from the start, and hence it is
not really useful in applications. There one would like to pass
through convexity exactly to prove quasi\hyph uniformizability, so
as to introduce and work in the compactified space.

The most common strategy is then that of adding some additional
conditions to the preorder such as the $C$-space and $I$-space
conditions \cite{priestley72} (compare with the definitions of
continuous and anti-continuous preorder in
\cite{mccartan71,burgess77}). A topological preordered space $E$ is
a $C$-space ($I$-space) if for every closed (open) subset $S$,
$d(S)$ and $i(S)$ are closed (resp.\ open).

The following theorem and proof are due to H.-P. K{\"u}nzi
\cite[Lemma 2]{kunzi92}. They are included for the reader
convenience.

\begin{theorem}
Every normal $T_1$-ordered $C$-space $(E,\mathscr{T},\le)$ is
convex.
\end{theorem}

\begin{proof}
Let $O$ be an open neighborhood of $x\in E$. The closed sets
$d(x)\setminus O$ and $i(x)\setminus O$ are disjoint. By normality
these sets can be separated by open sets, say $H_1$ and $H_2$, then
$d(x)\subset H_1\cup O$ and $i(x)\subset  H_2\cup O$. The set
$E\setminus (H_1\cup O)$ is closed and is disjoint from $d(x)$. By
the $C$-space assumption $i(E\setminus (H_1\cup O))$ is closed and
is disjoint from $d(x)$ thus $U=E\backslash i(E\setminus (H_1\cup
O))$ is an open decreasing set such that $d(x)\subset U\subset
H_1\cup O$. Analogously there is an open increasing set $V$ such
that $i(x)\subset V\subset H_2\cup O$. Thus $x\in (U\cap V)\subset
(H_1\cup O)\cap (H_2\cup O)\subset O$. Hence the space is convex.
\end{proof}
Unfortunately the $C$-space condition is too strong as not even
$\mathbb{R}^2$ with the product order is a $C$-space (consider the
increasing hull of the closed set $S=\{(x,y): x<0, y>0, y=-1/x\}$).

Concerning the $I$-space property we have the following
simplification.

\begin{theorem} \label{bfr}
Every  locally convex $I$-space $(E,\mathscr{T},\le)$ is convex.
\end{theorem}

\begin{proof}
Let $x\in E$ and let $O$ be an open neighborhood of $x$. By local
convexity there are a convex set $C$ and an open set $O'$ such that
$x\in O'\subset C\subset O$. Since $E$ is an $I$-space the sets
$V=i(O')$ and $U=d(O')$ are respectively open increasing and open
decreasing. Furthermore, $x\in U\cap V\subset d(C)\cap i(C)=C\subset
O$, which proves that $E$ is convex.
\end{proof}

In this connection, the next interesting result due to Burgess and
Fitzpatrick \cite[Cor.\ 4.4]{burgess77} is worth mentioning

\begin{theorem}
Every locally compact convex $T_2$-ordered $I$-space is completely
regularly ordered.
\end{theorem}

\begin{remark} \label{mix}
The $I$-space property is sometimes justified in applications. For
instance, in general relativity (see Sect.\ \ref{nod}) the closure
$\overline{J^+}$ of the causal relation in a causally continuous
spacetime provides a preorder which turns spacetime into a
topological closed preordered $I$-space.\footnote{Proof: we have
mentioned in Sect.\ \ref{nod} that under causal  continuity
$D^+=\overline{J^+}=K^+$ is antisymmetric, thus the spacetime is
stably causal and $(M,\mathscr{T},\overline{J^+})$ is a closed
ordered space. Let $O\subset M$ be an open set, and let $(x,y)\in
D^+$ where $x\in O$, so that $x\in \overline{J^{-}(y)}$. Pick $x'\in
I^{-}(x)\cap O$, then $x\in I^+(x')$ and hence $y\in I^+(x')$. Since
$I^+$ is open and $I^+\subset \overline{J^+}$, we conclude that a
neighborhood of $y$ is contained in $D^+(O)$ and hence that $D^+(O)$
is open.}

In this work we shall try to avoid as much as possible the
simplifying $C$-space and $I$-space assumptions, and we shall
instead impose weak conditions on the preorder and the topology in
order to attain convexity. We shall meet again the I-space
assumption at the end of this work, where it is used in connection
with strict quasi-pseudo-metrizability.
\end{remark}

We end the section with examples which show that a normally
preordered space need not be convex. An example can be found in
\cite[Example 4.9]{fletcher82}.

A locally compact $\sigma$-compact $T_2$-ordered space which is not
locally convex can be found in \cite[p.\ 59]{akin10b}. The next
example is particularly interesting because the topology  has nice
properties.

\begin{example} \label{bao}
Let $E=(0,1]\subset \mathbb{R}$ with the induced topology which we
denote $\mathscr{T}$. The topology is  particulary well behaved, it
is connected, metrizable, locally compact, $\sigma$-compact, second
countable.  Define on $E$ the order $\preceq$ through the following
increasing hulls
\begin{align*}
i(x)&=\{y\in (0,1]: x\le y\le 1-x\} &\textrm{ if } 0<x\le 1/2, \\
i(x)&=\{x\} &\textrm{ if } 1/2<x<1, \\
i(x)&=(0,1]  &\textrm{ if } x=1.
\end{align*}
With this definition the decreasing hulls are
\begin{align*}
d(x)&=\{y\in (0,1]: 0<y\le x \textrm{ or } y=1\} &\textrm{ if } 0<x\le 1/2, \\
d(x)&=\{y \in (0,1]: 0<y\le 1-x \textrm{ or } y=x \textrm{ or }
y=1\} &\textrm{ if } 1/2<x\le 1.
\end{align*}
It is easy to check that  $\preceq$ is reflexive, transitive and
antisymmetric and hence an order.  $E$ with this order is a
$T_2$-ordered space, indeed let $x_n\preceq y_n$ with $(x_n,y_n)\to
(x,y)$. If $x=1$ then necessarily as $i(x)=E$, $y\in i(x)$. If
$1/2<x<1$ then for sufficiently large $n$, $1/2<x_n<1$ thus
$y_n=x_n$ and then $y=\lim y_n=\lim x_n=x$ that is $y\in i(x)$. If
$0<x\le 1/2$ then we can assume, up to a subsequence, that either
for all $n$, $1/2<x_n<1$ (and hence $x=1/2$), or $0<x_n\le 1/2$. In
the former case $y_n=x_n$ and then $y=x=1/2$  thus $y\in i(x)$,
while in the latter case passing to the limit the equation $x_n\le
y_n\le 1-x_n$ we get $x\le y\le 1-x$ that is $y \in i(x)$ which
concludes the proof. Let us observe that $\mathscr{T}$ is second
countable and locally compact which implies that
$(E,\mathscr{T},\preceq)$ is a normally ordered space
\cite{minguzzi11f}. Nevertheless, convexity does not hold at $x=1$
and in fact even local convexity fails there because every convex
neighborhood of $1$ contains points `arbitrarily close to the lower
edge at 0'.
\end{example}


\section{From local convexity to convexity}

The mentioned examples of $T_2$-preordered locally compact
$\sigma$-compact spaces which are not convex are also non-locally
convex. This fact suggests that, perhaps, we could obtain convexity
by assuming local convexity plus some topological property. This is
indeed the case and in this section we shall prove that a locally
convex $T_2$-preordered locally compact $\sigma$-compact space  is
necessarily convex. This result is important because it is often
much easier to prove local convexity than convexity. The next two
sections will then show how to obtain local convexity for a large
class of topological preordered spaces.


We need to state the next two propositions which generalize to
preorders two corresponding propositions due to Nachbin \cite[Prop.\
4,5, Chap. I]{nachbin65}. Actually the proofs given by Nachbin for
the order case work unaltered. For this reason they are omitted.

\begin{proposition} \label{pr1}
Let $E$ be a  $T_2$-preordered space. For every compact set
$K\subset E$, we have $d(K)=D(K)$ and $i(K)=I(K)$, that is, the
decreasing and increasing hulls are closed.
\end{proposition}

\begin{proposition} \label{pr2}
Let $E$ be a  $T_2$-preordered  compact space. Let $F\subset V$
where $F$ is increasing and $V$ is open, then there is an open
increasing set $W$ such that $F\subset W\subset V$. An analogous
statement holds in the decreasing case.
\end{proposition}

We start with a convex analog to the previous proposition.

\begin{lemma} \label{prr}
Let $E$ be a  normally preordered  space, let $A$ be a closed
decreasing set and let $B$ be a closed increasing set. Finally, let
$S$ be a compact set and let $O$ be an open set such that $A\cap
B\cap S \subset O$, then there are an open decreasing set $U\supset
A$ and an open increasing set $V\supset B$, such that $D(U)\cap
I(V)\cap S \subset O$.
\end{lemma}

\begin{proof}
The set $K=S\backslash O$ being a closed subset of a compact set is
compact. Let $y\in K$, we know that $y \notin A$ or $y \notin B$. In
the former case there is an open increasing set $M_y \ni y$ and an
open decreasing set $U_y \supset A$ such that $U_y\cap
M_y=\emptyset$. If $y\in A$ (and hence $y\notin B$) there is an open
decreasing set $M_y \ni y$ and an open increasing set $V_y \supset
B$ such that $V_y\cap M_y=\emptyset$. Since $K$ is compact there are
some $y_i$, $i \in \Lambda$, $\Lambda=\{1,2,\cdots, n\}$, such that
the sets $M_{y_i}$ cover $K$. The index set $\Lambda$ splits into
the disjoint union of the two subsets $\Lambda_d$, $\Lambda_i$,
where $k\in \Lambda_d$ iff $y_k\notin A$. Let us define
$U'=\bigcap_{j \in \Lambda_d} U_{y_j}$ and $V'=\bigcap_{j\in
\Lambda_i} V_{y_j}$. The subsets $U', V'$ are such that $U'\supset
A$ and $V'\supset B$. Let us prove that $U'\cap V'\cap S\subset O$.
Indeed, suppose $z\in K=S\backslash O$,  then $z$ is contained in
some $M_{y_j}$, $j \in \Lambda$, that does not intersect $U'$ or
$V'$ depending on whether $y_j \notin A$ or not, thus $z \notin
U'\cap V'$ which implies $U'\cap V'\cap S\subset O$. By applying
preorder normality we find $U$ open decreasing set such that $A
\subset U \subset D(U) \subset U'$ and $V$ open increasing set such
that $B\subset V\subset I(V)\subset V'$, thus $D(U)\cap I(V)\cap
S\subset O$.
\end{proof}

\begin{lemma} \label{pri}
Let $E$ be a  $T_2$-preordered compact space, let $A$ be a closed
decreasing set and let $B$ be a closed increasing set. Finally, let
$O$ be an open set such that $A\cap B\subset O$. Then there are an
open decreasing set $U\supset A$ and an open increasing set
$V\supset B$, such that $D(U)\cap I(V)\subset O$.
\end{lemma}

\begin{proof}
Since $E$ is a $T_2$-preordered compact space it is normally
preordered \cite[Theor.\ 2.4]{minguzzi11f}. Setting $S=E$ the
desired conclusion follows from lemma \ref{prr}.
\end{proof}

It is well known that under Hausdorffness local compactness and
strong local compactness are equivalent. Every $T_2$-ordered space
is Hausdorff thus under antisymmetry these notions of local
compactness  coincide. We can actually prove that this equivalence
holds at a single point.

Let $S$ be a subspace of $E$. In the next theorems with ``on $S$''
we shall mean ``with respect to $S$ regarded as a subspace, namely
with its induced topology and induced preorder''. On $S$ the
increasing hull of a subset $H\subset S$ will be denoted $i_S(H)$
and analogously for the decreasing hull, $d_S(H)$, and for the
corresponding closure versions, $I_S(H)$ and $D_S(H)$.

\begin{proposition} \label{pbf}
Let $E$ be a $T_2$-preordered  space. If $[x]\subset E$ admits a
compact neighborhood then for every open set $O'\supset [x]$ there
is a compact neighborhood of $[x]$ contained in $O'$. In particular,
under antisymmetry at $x$, local compactness at $x$ implies strong
local compactness at $x$.
\end{proposition}

\begin{proof}
Let $K$ be a compact neighborhood of $[x]$, $[x]\subset \textrm{Int}
K$. Let $O'$ be an open set such that $[x]\subset O'$  and define
$O=O'\cap \textrm{Int} K$. Let $A=d(x)\cap K$, $B=i(x)\cap K$. Since
$[x]\subset O$, we have $d(x)\cap i(x)\subset O$ which implies
$A\cap B \subset O$. We work on the $T_2$-preordered compact space
$K$ and apply lemma \ref{pri}. There  are an open decreasing set
$U\supset A$ (on $K$) and an open increasing set $V\supset B$ (on
$K$), such that $D_K(U)\cap I_K(V)\subset O$. Since $U\cap V\subset
O\subset K$, the $K$-open set $U\cap V$ is actually open in $E$. The
set $D_K(U)\cap I_K(V)$ being a closed subset of $K$ is compact, and
containing $U\cap V$, it is actually a compact neighborhood of $x$
contained in $O$ and hence $O'$.
%
%
\end{proof}

%
%
%

\begin{lemma} \label{prj}
Let $E$ be a $T_2$-preordered  space, $S$ a compact subset, $O$  an
open set on $E$, $C$  a convex set on $E$, and $x$ a point in $S$,
such that $[x]_S \subset O\subset C\subset S$ where
$[x]_S=d_S(x)\cap i_S(x)$. Then $[x]=[x]_S$,  and there are an open
convex neighborhood (on $E$) of $[x]$ contained in $O$, and a
$c$-compact neighborhood (on $E$) of $[x]$ contained in $O$.
\end{lemma}

\begin{proof}
If  $z\in d(x)\cap i(x)$ then $z\in d(C)\cap i(C)=C\subset S$, thus
$z\in d_S(x)\cap i_S(x)=[x]_S$, that is $[x]=[x]_S$.

Let $N$ be a ($S$-)convex neighborhood of $[x]$ contained in $O$,
where convexity refers to the subspace $S$.  We are going to prove
that $N$ is convex in $E$. Indeed
\[
d(N)\cap i(N)\subset d(C)\cap i(C) \subset C,
\]
thus if $z\in d(N)\cap i(N)$ then $z\in S$ which implies
\[
d(N)\cap i(N)=d_S(N)\cap i_S(N)=N,
\]
by convexity of $N$ in $S$, thus $N$ is indeed convex in $E$.

Suppose that we prove the existence of an open convex neighborhood
$N$ of $[x]$ contained in $O$ in the $T_2$-preordered subspace $S$.
Since $N\subset O$ and $O\subset S$ is open in $E$, $N$ is open in
$E$ and also convex by the above argument.

Analogously, suppose that we prove the existence of a $c$-compact
neighborhood $N$ of $[x]$ contained in $O$ in the $T_2$-preordered
subspace $S$ with the additional property that it contains an open
convex (in the subspace $S$) neighborhood $N'$ of $[x]$ contained in
$O$.  Since $N$ is compact on $S$ it is compact on $E$. The equation
$d(N)\cap i(N)=d_S(N)\cap i_S(N)=N$ proves that $N$ is closed in
$E$, and that it is a $c$-compact set in $E$ (recall Prop.
\ref{pr1}). Since it contains $N'$ which contains $[x]$ it is a
$c$-compact neighborhood of $[x]$.

Thus for the remainder of the proof we can work in the
$T_2$-preordered compact subspace $S$. Let $A=d_S(x)$ and
$B=i_S(x)$, so that $A\cap B=[x]_S\subset O$. Lemma \ref{pri} proves
that there are an open convex neighborhood $U\cap V$ of $[x]$
contained in $O$ (according to the subspace $S$), and a $c$-compact
neighborhood $D_S(U)\cap I_S(V)$ of $[x]$ contained in $O$
(according to the subspace $S$), which finishes the proof.
\end{proof}

\begin{lemma} \label{gvh}
If local convexity holds at $x\in E$ then $[x]$ is compact and every
open neighborhood of $x$ is also an open neighborhood of $[x]$.
\end{lemma}

\begin{proof}
Let $O$ be an open neighborhood of $x$ and let $C$ be a convex set
such that $x\in C\subset O$, then $[x]=d(x)\cap i(x)\subset d(C)\cap
i(C)=C \subset O$, thus $O$ is also an open neighborhood for $[x]$.
Let us consider an open covering of $[x]$, then there is some open
set of the covering which includes $x$ and hence $[x]$, thus every
open covering admits a subcovering of only one element.
\end{proof}

\begin{proposition} \label{nic}
Let $E$ be a $T_2$-preordered space. If $E$ is locally compact and
locally convex at $x\in E$, then the topology at $x$ admits a base
of $c$-compact neighborhoods, and a base of open convex
neighborhoods (that is, weak convexity holds at $x$).
\end{proposition}

\begin{proof}
Let $N$ be any open neighborhood of $x$. By local compactness there
are a compact set $S$ and an open set $O'$ such that $x\in O'\subset
S\cap N$. By local convexity there is a convex set $C$ and an open
set $O$ such that $x\in O\subset C\subset O'$. By local convexity
(lemma \ref{gvh}) $[x]\subset O$. By lemma \ref{prj} there are an
open convex neighborhood of $x$ contained in $O$ (and hence $N$) and
a $c$-compact neighborhood of $x$ contained in $O$ (and hence $N$).
\end{proof}

\begin{corollary} \label{cid}
Every locally compact and locally convex closed preordered space is weakly convex.
\end{corollary}

\begin{lemma} \label{hag}
Let $E$ be a normally preordered space, $S$ a compact subset, $A$ a
closed decreasing set on $S$ and $B$ a closed increasing set on $S$.
Further, let $O$ be an open and convex set on $E$ (not necessarily
contained in $S$) such that
\[
A\cap B\subset O,
\]
then there are an open decreasing set $U$ on $S$ and an open
increasing set $V$ on $S$, such that $A\subset U$, $B\subset V$, and
\[
d(D_S(U))\cap i(I_S(V)) \subset O.
\]
\end{lemma}

\begin{proof}
Since $S$ is a subspace and the $T_2$-preorder property is
hereditary, the subset  $S$, with the induced preorder and topology,
is a $T_2$-preordered space and, being compact, it is a normally
preordered space \cite{minguzzi11f}.

By lemma \ref{pri} and by preorder normality of $S$ there are
$\hat{U}, \check{U}\subset S$, open decreasing sets on $S$ and
$\hat{V},\check{V}\subset S$, open increasing sets on $S$ such that
$\check{U}\cap \check{V}\subset O\cap S$ and
\begin{align*}
A\subset \hat{U}\subset D_S(\hat{U}) \subset \check{U},\qquad
B\subset \hat{V}\subset I_S(\hat{V})\subset \check{V}.
\end{align*}
The set $A\backslash \hat{V}\subset S$ is closed on $S$ and hence
compact on both $S$ and $E$, decreasing on $S$ and disjoint from $B$
thus, $d(A\backslash \hat{V}) \cap i(B)=\emptyset$ where
$d(A\backslash \hat{V})$ is closed decreasing on $E$ and $i(B)$ is
closed increasing on $E$. By preorder normality of $E$ there are
$\tilde{U}_A$ open decreasing on $E$ and $\tilde{V}_A$ open
increasing on $E$, such that
\[d(A\backslash \hat{V})\subset \tilde{U}_A, \quad i(B)\subset
\tilde{V}_A\quad \textrm{ and } \quad \tilde{U}_A\cap
\tilde{V}_A=\emptyset.\]

Analogously, $B\backslash \hat{U}$ is closed on $S$, hence compact
on both $S$ and $E$, increasing on $S$ and disjoint from $A$,
$i(B\backslash \hat{U})$ is closed increasing in $E$, $d(A)$ is
closed decreasing on $E$, we have $i(B\backslash \hat{U})\cap
d(A)=\emptyset$ and we find $\tilde{U}_B$ open decreasing on $E$,
$\tilde{V}_B$ open increasing on $E$ such that
\[d(A)\subset \tilde{U}_B, \quad i(B\backslash \hat{U})\subset
\tilde{V}_B\quad  \textrm{ and } \quad \tilde{U}_B\cap
\tilde{V}_B=\emptyset.\]

Let us define the open subsets of $S$
\begin{align*}
P_A&=(O\cup \tilde{U}_A) \cap \tilde{U}_B\cap \check{U}, \\
P_B&=(O\cup \tilde{V}_B) \cap \tilde{V}_A\cap \check{V}.
\end{align*}
We have $A \subset P_A$ because $A\subset \tilde{U}_B\cap \check{U}$
and if $x\in A\backslash \hat{V}$ then $x\in \tilde{U}_A$ while if
$x \in \hat{V}\cap A\subset \hat{V}\cap \hat{U}\subset O$ we have
$x\in O$. Analogously, $B\subset P_B$.

Let us prove that $d(P_A)\cap i(P_B)\subset O$.  If $z \in
d(P_A)\cap i(P_B)$ then there are $x\in P_A$ and $y\in P_B$, such
that $y \le z \le x$. The possibility $x\in \tilde{U}_A$ is excluded
because $y\in P_B\subset \tilde{V}_A$, and $\tilde{U}_A\cap
\tilde{V}_A=\emptyset$. Analogously, $y \in \tilde{V}_B$ is excluded
because  $x \in P_A\subset \tilde{U}_B$ and $\tilde{V}_B\cap
\tilde{U}_B=\emptyset$. Thus
\begin{align*}
& x\in P_A\backslash \tilde{U}_A\subset O\cap \tilde{U}_B\cap
\check{U},
\\
& y\in P_B\backslash \tilde{V}_B\subset O\cap \tilde{V}_A\cap
\check{V}.
\end{align*}
Since $x,y \in O$ which is convex we obtain $z\in O$, that is
$d(P_A)\cap i(P_B)\subset O$.

Using Prop.\ \ref{pr2}, since $P_A$ is open in $S$ and $A$ is
decreasing in $S$ there is $U'$ open decreasing on $S$ such that
$A\subset U'\subset P_A$ and applying preorder normality of $S$
there is $U$, open decreasing on $S$, such that $A \subset U\subset
D_S(U)\subset U'\subset P_A$. Analogously, we find $V,V',$ open
increasing sets on $S$, such that $B \subset V\subset I_S(V)\subset
V'\subset P_B$. We have
\[
d(D_S(U)) \cap i(I_S(V)) \subset d(P_A)\cap i(P_B)\subset O.
\]
\end{proof}


\begin{lemma} \label{pry}
Let $E$ be a  $T_2$-preordered $k_\omega$-space, $x\in E$,  and let
$O$ be an open and convex neighborhood of $x$. Then there are an
open decreasing set $U$ and  open increasing set $V$, such that
$x\in U\cap V\subset O$.
\end{lemma}

\begin{proof}
We already know that $E$ is normally preordered \cite{minguzzi11f}.
Since $O$ is convex $[x]\subset O$.  Let $K_i$, $K_i\subset
K_{i+1}$, be an admissible sequence for the $k_\omega$-space $E$.
Without loss of generality we can assume $x\in K_1$. Each $K_i$
endowed with the induced topology and preorder is a compact
$T_2$-preordered space.


Let $A_1=d(x)\cap K_1$, $B_1=i(x)\cap K_1$. We have that $A_1$ is
closed decreasing in $K_1$, $B_1$ is closed increasing in $K_1$ and
$A_1\cap B_1\subset O$. Since $K_1$ is compact, by lemma \ref{hag}
(with $S=K_1$) we can find $U_1\supset A_1$, open decreasing set in
$K_1$, and $V_1\supset B_1$, open increasing set in $K_1$, such that
$d(D_1(U_1))\cap i(I_1(V_1))\subset O$, where $D_1$ and $I_1$ are
the closed-hull maps of $K_1$. Observe that $D_1(U_1)$ and
$I_1(V_1)$ being closed subsets of $K_1$ are compact in $E$. We
define $A_2=d_2(D_1(U_1))$ and $B_2=i_2(D_1(V_1))$, where $A_2$ is
clearly closed decreasing in $K_2$, and $B_2$ is closed increasing
in $K_2$. We have $A_2\cap B_2\subset d(D_1(U_1))\cap
i(I_1(V_1))\subset O$. We can proceed applying again  lemma
\ref{hag} with $S=K_2$. Thus proceeding inductively, given
$A_i,B_i\subset K_i$, $A_i\cap B_i\subset O$ we find $U_i$, $V_i$
respectively open decreasing and open increasing subsets of $K_i$
such that $U_i\supset A_i$, $V_i\supset B_i$, $d(D_i(U_i))\cap
i(I_i(V_i))\subset O$, and define $A_{i+1}=d_{i+1}(D_i(U_i))$ and
$B_{i+1}=i_{i+1}(D_i(V_i))$.

Note that $V_j \subset {B}_{j+1} \subset V_{j+1}$ and analogously,
$U_j\subset U_{j+1}$.

Let us define the sets $U=\bigcup_{j=1} U_j$ and $V=\bigcup_{j=1}
V_j$. The set $V$ is open because $V\cap K_s=\bigcup_{j\ge 1}
(V_j\cap K_s)=\bigcup_{j\ge s} (V_j\cap K_s)$, and the set
$V_j\subset K_j$ is open in $K_j$ so that, since for $j\ge s$, $K_s
\subset K_j$, $V_j\cap K_s$ is open in $K_s$ and so is the union
$V\cap K_s$. The
 $k_\omega$-space property implies that $V$ is open. Analogously, $U$ is
open.

Let us prove that $V$ is increasing. Let $w\in V$ then there is some
$j\ge 1$ such that $w\in V_j \subset K_j$. Let $y\in i(w)$, then we
can find some $r\ge j$ such that $y \in K_r$. Since $V_j \subset
V_r$, $w\in V_r$, and since $V_r$ is increasing on $K_r$, $y \in
V_r$ thus $y\in V$. Analogously, $U$ is decreasing. Finally, if
$z\in U\cap V$ then there are some $j,k \ge 1$ such that $z \in
U_j\cap V_k$ and setting $r=\max(j,k)$, $z\in U_r\cap V_r$ thus
\[
U\cap V\subset \bigcup_{r=1} (U_r\cap V_r)\subset  O'\subset O.
\]
\end{proof}
%
%
%
%

As an immediate consequence we obtain the desired result.

\begin{theorem}
Every weakly convex $T_2$-preordered $k_\omega$-space is a convex
normally preordered space (and hence quasi-uniformizable).
\end{theorem}

\begin{remark} \label{vux}
Actually we proved something more, namely that a $T_2$-preordered
$k_\omega$-space which is weakly convex at $x$ is convex at $x$.
Thus, by Prop.\ \ref{nic}, in a $T_2$-preordered $k_\omega$-space
$E$, if local convexity and local compactness hold at $x$, then
convexity holds at $x$.
\end{remark}

\begin{corollary} \label{bak}
Every locally convex $T_2$-preordered locally compact
$\sigma$-compact space is a convex normally preordered space (and
hence quasi-uniformizable).
\end{corollary}

\begin{proof}
Every locally compact $\sigma$-compact space is a $k_\omega$-space,
and under local compactness local convexity and weak convexity are
equivalent (Cor.\ \ref{cid}).
\end{proof}

\section{Convexity of $k$-preserving spaces}

The next definition is inspired by the property of global
hyperbolicity in Lorentzian geometry, see Sect.\ \ref{nod}.

\begin{definition}
A $T_2$-preordered space $E$ is  {\em $k$-preserving }  if every
compact set $K\subset E$ has a compact convex hull $d(K)\cap i(K)$.
\end{definition}

\begin{proposition} \label{nid}
Let $E$ be a $T_2$-preordered space. If the topology does not
distinguish the points of $[x]$ (e.g.\ if $E$ is locally convex at
$x$ or antisymmetry holds at $x$) and $x$ admits a
 $c$-compact neighborhood, then $x$ admits a base of
$c$-compact neighborhoods and, moreover, $E$ is weakly convex at
$x$.
\end{proposition}

\begin{proof}
Let $K$ be a $c$-compact neighborhood of $x$ and hence $[x]$, and
let $O\ni x$ be an open neighborhood of $x$ and hence $[x]$ which we
can assume contained in $K$. We have to show that there is a compact
$c$-set neighborhood $K'$ of $[x]$ such that $K'\subset O$, and
analogously in the convex open neighborhood case. It suffices to
apply lemma \ref{prj} with $S:=K$, $C:=K$. Observe that by local
convexity (lemma \ref{gvh}) or antisymmetry at $x$, if $O$ is any
open neighborhood of $x$ we have $[x]\subset O$.
\end{proof}

Clearly a compact $T_2$-ordered space is $k$-preserving (Prop.\
\ref{pr1}).
%
%
We know that the compact $T_2$-ordered spaces are convex
\cite{nachbin65}. We have the following interesting generalization
%
%
%

\begin{theorem} \label{bdd}
Every $T_2$-preordered  $k$-preserving $k_\omega$-space is convex at
every point $x$ such that (i) the topology does not distinguish
different points of [x], (ii) local compactness holds at $x$ (e.g.\
wherever it is locally compact and antisymmetric).

In particular, every $k$-preserving $T_2$-ordered locally compact
$\sigma$-compact space is convex (and hence quasi-uniformizable).
\end{theorem}

\begin{proof}
Every $T_2$-preordered $k_\omega$-space  is normally preordered
\cite{minguzzi11f}. By assumption there is a compact neighborhood
$K$ of $[x]$. The set $d(K)\cap i(K)$ is a $c$-compact neighborhood
of $x$. By Prop.\ \ref{nid}, weak convexity holds at $x$, and by
remark \ref{vux} convexity holds at $x$.
\end{proof}

\begin{remark}
Actually the $k$-preserving property could be dropped provided we
replace (ii) with the requirement that the point $x$ has a c-compact
neighborhood, or that local compactness holds at $x$ and the
$k$-preserving property holds {\em locally}.
\end{remark}

\section{Compactly generated $T_2$-preorders}

In this section we study sufficient conditions for local convexity.
The main idea is to consider preorders which, intuitively, are
generated by relations which are limited, in the sense that do not
connect arbitrarily 'far away' points (compactness is used to give a
rigorous meaning to this concept). Thus we shall be basically
concerned with topological preordered spaces for which both topology
and preorder are generated from local information.

For this type of preorder and for a locally compact space, given two
related `far away' points $p,q,$ there is some point $r$, $p\le r
\le q$, at `reasonable distance' but not too close to the original
point $p$. From that it is possible to show that if local convexity
is violated at $p$ then, by a limiting argument, some point $r'\ne
p$ exists such that $p\le r'\le p$ and hence antisymmetry is
violated at $p$. This strategy has been used in mathematical
relativity theory to prove that the $K^+$ relation (the smallest
closed preorder containing the causal relation $J^+$) is locally
convex \cite[Lemma 16]{sorkin96} \cite[Lemma 5.5]{minguzzi07}.

\begin{definition} \label{jdp}
A  $T_2$-preordered space $(E,\mathscr{T},\le)$ is a
$k$-$T_2$-preordered space (read `compactly generated
$T_2$-preordered space') if there is a relation $R\subset G(\le)$
such that\footnote{Compare with the definition of +proper relation
$R$,
 and relation $\mathcal{G}R$ given in \cite{akin10b}.}
\begin{itemize}
\item[(i)] for every  compact set $K$ the set $\overline{R(K)}$ is compact,
\item[(ii)] the preorder $\le$ is the smallest closed preorder
containing $R$.
\end{itemize}
We shall also say that $\le$ is a compactly generated preorder.
\end{definition}

Note that in (ii) the smallest closed preorder exists because the
family of closed preorders containing $R$ is non-empty as $E\times
E$ is a closed preorder which contains $R$. Note that if $R$
satisfies (i)-(ii) then also $\Delta\cup R$ satisfies them, thus $R$
can be chosen reflexive.

\begin{remark} \label{nsg}
For applications in which $E$ is locally compact it is useful to
observe that the condition
\begin{itemize}
\item[(i')] every point $x\in E$ admits a closed and compact neighborhood $F(x)$ such
that $\overline{R(F)}$ is compact,
\end{itemize}
implies (i), and thus a space $E$ satisfying (i') and (ii) is
compactly generated. Note that if $R$ satisfies (i')-(ii) then also
$\Delta\cup R$ satisfies them, thus $R$ can be chosen reflexive.
\end{remark}

%
%


\begin{proposition} \label{nco}
If $(E,\mathscr{T},\le)$ is a  $T_2$-preordered compact space, then
$\le$ is compactly generated.
\end{proposition}

\begin{proof}
The conditions in the definition of compactly generated preorder are
satisfied taking $R=G(\le)$.
\end{proof}

The next result is worth mentioning although we shall not use it.

\begin{theorem}
Let $(E,\mathscr{T},\le)$ be a $k$-$T_2$-preordered space, and let
$R$ be a reflexive relation as in  definition \ref{jdp}. The set of
continuous isotone functions for $R$ coincides with the set of
continuous isotone functions for $\le$.

\end{theorem}

\begin{proof}
If $f$ is a continuous isotone function for $\le$ and $(x,y)\in R$
we have, since $R\subset G(\le)$, $(x,y)\in G(\le) \Rightarrow
f(x)\le f(y)$ thus $f$ is a continuous isotone function for $R$. If
$f$ is a continuous isotone function for $R$ the relation
$R_f=\{(w,z): f(w)\le f(z)\}$ is a closed preorder  containing $R$
thus $G(\le)\subset R_f$ which implies $x\le y\Rightarrow f(x)\le
f(y)$, that is, $f$ is a continuous isotone function for $\le$.
\end{proof}

This result is interesting because in those cases in which $E$ is
also normally preordered (the $k_\omega$-space condition suffices
\cite{minguzzi11f}) this set of continuous isotone functions for $R$
allows us to recover $\le$, that is, $x\le y$ iff for all continuous
isotone functions $f: E\to [0,1]$, we have $f(x)\le f(y)$.

\begin{remark} \label{hnx}
It is worth to mention a recent work by Akin and Auslander on
recurrence problems and compactifications in dynamical systems
\cite{akin10b}. This section is very much related with their work,
although we followed a different line of reasoning inspired by
results in topological preordered spaces and relativity theory. In
their paper they assume that $E$ is a separable locally compact
metric space \cite[p.\ 50]{akin10b}, while in our work second
countability and Hausdorffness are not assumed, and local
compactness is used only where it is strictly needed. We do not use
compactification arguments as in their article.


We usually work with  a reflexive relation $R$  because this is the
interesting case from the topological point of view, as the elements
of a quasi-uniformity contain the diagonal. Furthermore, the
application to cone structures  seems to require a reflexive $R$.
Observe that if $R$ is reflexive then the generalized recurrent set
mentioned in \cite[Theor.\ 11]{akin10b} is the whole space. Our
theorem \ref{bdr} will be similar but stronger than  their
\cite[Theor.\ 14]{akin10b}.

We find that our terminology concerning {\em compactly  generated}
$T_2$-preordered spaces is more appropriate, since relations do
generalize functions but the term {\em proper} is used for maps such
that the inverse images of compact subsets are compact, while we do
not take any inverse here. Maps which send compact set to compact
sets are sometimes called {\em compact}. Finally, observe that our
terminology places the accent on $\le$ rather than $R$. In
applications there is often a natural choice for $R$ but,
mathematically, it could be chosen with some freedom.

\end{remark}

\subsection{Some examples of compactly generated preorders}
\label{jej}

 Most closed preorders appearing in applications are
compactly generated. We give some examples proving conditions
(i)-(ii) or (i')-(ii) of remark \ref{nsg}.

\begin{example} \label{bws}
Let us recall that in a spacetime $(M,g)$ (see Sect.\ \ref{nod}) the
relation $K^+$ is by definition the smallest closed and transitive
relation containing $J^+$. Let $F_\alpha$ be a locally finite closed
and compact covering of $M$ (it exist because of local compactness
 and \cite[Theor.\ 20.7]{willard70}) and let $R=\cup_{\alpha} J^+\cap (F_\alpha \times F_\alpha)$.
Since each $F_\alpha$ is intersected only by a finite number of
$F_\beta$,  $\overline{R(F_\alpha)}$ is compact. Thus if $C$ is a
compact set, $\overline{R(C)}$ is compact.

Clearly, $J^+$ is the smallest transitive relation containing $R$,
thus $K^+$ is the smallest closed and transitive relation containing
$R$. We conclude that $K^+$ is a $k$-$T_2$-preorder for which $R$ is
a generating relation. As mentioned in Sect.\ \ref{nod}, it
coincides with the causal relation in causally simple spacetimes and with its closure in stably causal spacetimes.

\end{example}

\begin{example}\label{crg}
Let $E$ be a Hausdorff, connected, paracompact ($C^{r+1}$, $r\ge 0$)
manifold and let $v: E\to TE$ be a ($C^r$, $r\ge 0$) vector field.
We write $(x,y)\in J$ if there is some integral curve of $v$ which
connects $x$ to $y$ in the forward direction. Let $F_\alpha$ be a
locally finite closed and compact covering of $E$, and let
$R=\cup_{\alpha} J\cap (F_\alpha \times F_\alpha)$. Clearly, $J$ is
the smallest transitive relation containing $R$. One is interested
in the smallest closed and transitive relation containing $J$,
denoted $\mathcal{G}J$ by some authors \cite{akin10b}, which is
therefore the smallest closed and transitive relation containing
$R$. Arguing as in example \ref{bws}, we obtain that $\mathcal{G}J$
is a $k$-$T_2$-preorder for which $R$ is a generating relation.
\end{example}

\begin{example}
Let $E$ be a Hausdorff, connected, compact manifold and let $f: E\to
E$ be a continuous map. We write $(x,y)\in J$ if there is some
integer $k\ge 0$ such that $y=f^k(x)$. Let $R=\{(x,f(x)), x\in E\}$.
One is interested in the smallest closed and transitive relation
containing $R$, which is clearly a $k$-$T_2$-preorder for which $R$
is a generating relation.
\end{example}

%
%
%
%
%

\subsection{Antisymmetry and local convexity}

The next two proofs generalize to the topological preordered case,
ideas contained in \cite[Lemma 5.3, 5.5]{minguzzi07}
\cite{sorkin96}.

\begin{proposition} \label{jhg}
Let $(E,\mathscr{T},\le)$ be a  $k$-$T_2$-preordered space, let $R$
be a reflexive generating relation as in definition \ref{jdp} and let $K$
be a compact set. If $x \le z$ with $x\in \textrm{Int}(K)$ and
$z\notin \overline{R(K)}$, then there is $y\in
\overline{R(K)}\backslash \textrm{Int}(K)$ such that $x \le y \le
z$.
\end{proposition}

\begin{proof}
Let us consider the relation, where $O=\textrm{Int}(K)$,
\begin{align*}
B=\{(x,z)\in G(\le):& \ \textrm{if} \ ``x\in O \textrm{ and }
z\notin \overline{R(K)} \, \textrm{''} \textrm{ then }
``\textrm{there is } y\in \overline{R(K)}\backslash O \\ & \textrm{
such that } x \le y \le z \textrm{''} \}.
\end{align*}
Suppose we prove that it is closed, reflexive, transitive and that
$R\subset B \subset G(\le)$. From the minimality property in the
definition of $\le$, we have $G(\le)\subset B$, thus $B=G$ which is
the thesis.

The inclusion $B\subset G$ is trivial, let us prove $R\subset B$. If
$(x,z)\in R$ then, by the definition of $\le$, $(x,z)\in G$. In the
definition of $B$ the hypothesis $``x\in O \textrm{ and } z\notin
\overline{R(K)} \, \textrm{''}$ is necessarily false because if
$x\in O$ then $z\in R(O)\subset R(K)\subset \overline{R(K)}$. As the
hypothesis is false the implication in the definition of $B$ is
true, thus $(x,z)\in B$ which proves $R\subset B$. Since $R$ is reflexive $B$ is reflexive.

%
%

Let us prove closure. Let $(x,z)\in \overline{B}$. If $x\notin O$ or
$z\in \overline{R(K)}$ then $(x,z) \in B$ because the hypothesis
``if $x\in O$ \textrm{ and } $z\notin \overline{R(K)}$'' in the
definition of $B$ is false and so the implication in the definition
of $B$ is true. Thus we can consider the case $x\in O$ and $z\notin
\overline{R(K)}$. Let $O_x \ni x, O_z\ni z$ be any open sets with
$O_x\subset O$, $O_z\subset E\backslash \overline{R(K)}$. By
assumption we can find $x'\in O_x$, $z'\in O_z$, $(x',z')\in B$.
Since $x' \in O$ and $z'\notin \overline{R(K)}$ there is $y'\in
\overline{R(K)}\backslash O$ such that $x'\le y' \le z'$. This means
that $i(O_x)\cap d(O_z)\cap \overline{R(K)}\backslash O$, with $O_x$
and $O_z$ running on the open neighborhoods of $x$ and $z$, gives a
family of non-empty sets with the finite intersection property (in
fact they are a base for a filter). As $\overline{R(K)}\backslash O$
is compact the associated filter admits a cluster point $y\in
\overline{R(K)}\backslash O$, i.e.\ every neighborhood of $y$
intersects $i(O_x)\cap d(O_z)$ for every open sets $O_x\ni x$,
$O_z\ni z$. But if it were $x\nleq y$ then, by the closure of $G$
\cite[Prop.\ 1]{nachbin65}, there would be a neighborhood of $O_x$
of $x$ and $O_y$ of $y$ such that $i(O_x)\cap d(O_y)=\emptyset$,
since it does not hold we infer
 $x \le y$ and analogously $y \le z$. Thus $(x,z)\in B$.




Let us prove transitivity. Let $(x,w)\in B$ and $(w,z)\in B$. If
``$x\in O$ and $z\notin \overline{R(K)}$'' is false there is nothing
to prove because, as the hypothesis in the implication defining $B$
is false, $(x,z)\in B$. If `$x\in O$ and $z\notin \overline{R(K)}$''
is true and $w \in \overline{R(K)}\backslash O$ we have gain
$(x,z)\in B$ (set $y=w$), thus let us assume $w\notin
\overline{R(K)}\backslash O$ so that it either belongs to $O$ or to
$E\backslash \overline{R(K)}$. In the former case since $(w,z)\in B$
we infer that there is $y\in \overline{R(K)}\backslash O$ such that
$x \le w\le y\le z$. In the latter case since $(x,w)\in B$ we infer
that there is $y\in \overline{R(K)}\backslash O$ such that $x \le
y\le w\le z$. We conclude $(x,z)\in B$.
\end{proof}

\begin{theorem} \label{bwx}
Let $(E,\mathscr{T},\le)$ be a  $k$-$T_2$-preordered space. Let
$x\in \textrm{Int} K$ where $K$ is a compact set. There is an open
convex neighborhood $C$ of $[x]$  such that $C\subset \textrm{Int}
K$ or there is a point $p\notin \textrm{Int} K$ such that $x\le p\le
x$.
\end{theorem}

\begin{proof}

The set $[x]$ is closed and we can assume that it is
contained in $\textrm{Int}(K)$, otherwise we have finished.


Let $R$ be a reflexive relation as in the definition \ref{jdp} and
let us set $O=\textrm{Int}(K)$. Either (i) there is a neighborhood
$N$ of $[x]$ such that $d(N)\cap i(N)$ is contained in the  compact
set $\overline{R(K)}$ or (ii) for every neighborhood $M$ of $[x]$,
$(d(M)\cap i(M))\backslash \overline{R(K)}\ne \emptyset$.

In  case (ii), if $z \in (d(M)\cap i(M))\backslash \overline{R(K)}$
with $M\subset O$  neighborhood of $[x]$, then by Prop.\ \ref{jhg}
since $z \notin \overline{R(K)}$ and $w\le z$  for some $w\in
M\subset O$, we have that there is $y \in \overline{R(K)}\backslash
O$ such that $w \le y \le z \le q$ for some $q\in M$. Thus for every
neighborhood $M\subset O$ of $[x]$, the set $(d(M)\cap i(M))\cap
(\overline{R(K)}\backslash O)$ is non-empty. Observe that varying
$M$ we obtain a family of sets which satisfies the finite intersection
property. As $\overline{R(K)}\backslash O$ is compact the associated
filter admits a cluster point $p\in \overline{R(K)}\backslash O$,
i.e.\ every neighborhood of $p$ intersects $d(M)\cap i(M)$ for every
neighborhood $M$ of $[x]$. But if it were $x\nleq p$ then by the
closure of $G(\le)$ and the compactness of $[x]$ there would be a
neighborhood  $O_x$ of $[x]$ and $O_p$ of $p$ such that $i(O_x)\cap
O_p=\emptyset$, since it does not hold we infer $x\le p$ and
analogously $p \le x$.

In case (i)  there is a neighborhood $N$ of $[x]$ such that
$d(N)\cap i(N)$ is contained in the  compact set $\overline{R(K)}$.
Assume that for every neighborhood $Y\subset N$ of $[x]$, $d(Y)\cap
i(Y)\cap [\overline{R(K)}\backslash O]\ne \emptyset$. Varying $Y$ we
get a family of non-empty subsets of the compact set
$\overline{R(K)}\backslash O$ which satisfies the finite intersection
property. There is a cluster point $p\in \overline{R(K)}\backslash
O$ thus, arguing as above, $x \le p \le x$.

If for some neighborhood $Y'\subset N$ of $[x]$, $d(Y')\cap
i(Y')\subset O$ we have that the set $C'= d(Y')\cap i(Y')$ is a
convex neighborhood of $[x]$ contained in $O$. Let $O'$ be an open
set such that $[x]\subset O' \subset C'\subset K$. Lemma \ref{prj} with $S=K$, implies the existence of an open convex
 neighborhood $C$ of $[x]$ contained in $O'$ and hence in $O$.
\end{proof}


\begin{theorem} \label{loi}
Every $k$-$T_2$-preordered space is weakly convex at every point $x$
such that (i) the topology does not distinguish different points of
[x], (ii) local compactness holds at $x$ (e.g.\ wherever it is
locally compact and antisymmetric).
\end{theorem}

\begin{proof}
Let $O$ be an open neighborhood of $x$ and hence $[x]$. By local
compactness there is a compact neighborhood of $x$ and hence $[x]$.
By proposition \ref{pbf} there is a compact neighborhood $K$ of
$[x]$ contained in $O$. By theorem \ref{bwx} there is an open convex
neighborhood of $x$ contained in $K$ and hence $O$.
\end{proof}

\begin{corollary} \label{nja}
Every  $k$-$T_2$-ordered locally compact space is weakly convex.
\end{corollary}

\begin{theorem}
Every $k$-$T_2$-preordered $k_\omega$-space is convex at every point
$x$ such that (i) the topology does not distinguish different points
of [x], (ii) local compactness holds at $x$ (e.g.\ wherever it is
locally compact and antisymmetric).
\end{theorem}

\begin{proof}
By theorem \ref{loi} weak convexity holds at $x$. By remark
\ref{vux} convexity holds at $x$.
\end{proof}

\begin{corollary} \label{bdr}
Every locally compact $\sigma$-compact $k$-$T_2$-ordered space  is
convex (and since they are normally ordered they are quasi\hyph
uniformizable).
\end{corollary}

\begin{proof}
Under local compactness the $k_\omega$-space property and
$\sigma$-compactness are equivalent.
\end{proof}


With reference to the example of compactly generated preorder given
by Example \ref{bws} (see also Sect. \ref{nod}) we have the
following consequence.

\begin{theorem} \label{nix}
Let $(M,g)$ be a stably causal spacetime, let $\mathscr{T}$ be the
manifold topology and let $J^+_S$ be the Seifert relation,  then
$(M,\mathscr{T}, J^+_S)$ is quasi-uniformizable and hence admits the
Nachbin compactification.
\end{theorem}

With reference to Example \ref{crg} we obtain:

\begin{theorem}
Let $E$ be the dynamical system whose flow is generated by a vector
field described in Example \ref{crg}, let $\mathscr{T}$ be the
manifold topology, and let $J$ be the reflexive relation there
defined. If $\mathcal{G}J$ is antisymmetric then
$(E,\mathscr{T},\mathcal{G}J)$ is quasi-uniformizable and hence
admits the Nachbin compactification.
\end{theorem}

\section{Quasi-pseudo-metrizability}

A {\em quasi-pseudo-metric} \cite{kelly63,patty67}  on a set $X$ is
a function $p:X\times X \to [0,+\infty)$ such that for $x,y,z\in X$
\begin{itemize}
\item[(i)] $p(x,x)= 0$,
\item[(ii)] $p(x,z)\le p(x,y)+p(y,z)$.
\end{itemize}
The quasi-pseudo-metric is called {\em quasi-metric} \cite{wilson31}
if (i) is replaced with (i'): $p(x,y)= 0$ iff $x=y$. Other
variations exist in the literature. The {\em Albert's quasi-metric}
\cite{albert41} is a special type of quasi-pseudo-metric which is
obtained replacing (i) with (i'') $p(x,y)=p(y,x)=0$ iff $x=y$.

The quasi-pseudo-metric is called {\em pseudo-metric} if
$p(x,y)=p(y,x)$. If a quasi-metric is such that $p(x,y)=p(y,x)$,
then it is a {\em metric} in the usual sense. If $p$ is a
quasi-pseudo-metric then $p^{-1}$, defined by $p^{-1}(x,y)=p(y,x)$,
is a quasi-pseudo-metric called {\em conjugate} of $p$. Each
quasi-pseudo-metric  $p$ generates a topology whose base is given by
the $p$-balls, $B^p_\epsilon(x)=\{y: p(x,y)< \epsilon\}$.

A topological preordered space $(E,\mathscr{T},\le)$ is {\em
quasi-pseudo-metrizable} if  there is a pair of conjugate
quasi-pseudo-metrics $p,q$, called {\em admissible}, such that
$\mathscr{T}$ is the topology generated by the pseudo-metric $p+q$
(equivalently $p\vee p^{-1}$), and the graph of the preorder is
given by $G(\le)=\{(x,y): p(x,y)=0\}$.

In the literature on bitopological spaces \cite{kelly63,patty67} a
bitopological space $(X,\mathscr{P},\mathscr{Q})$ is {\em
quasi-pseudo-metrizable} if there is a quasi-pseudo-metric $p$ such
that $p$ generates $\mathscr{P}$ and $p^{-1}$ generates
$\mathscr{Q}$.

A topological preordered space $(E,\mathscr{T},\le)$ is {\em
strictly quasi-pseudo-metrizable} if it is convex semiclosed
preordered and there is a pair of conjugate quasi-pseudo-metrics
$p,q$, such that the topology associated to $p$ is the upper
topology $\mathscr{T}^\sharp$, and the topology associated to $q$ is
the lower topology $\mathscr{T}^\flat$. In other words,  according
to our terminology $(E,\mathscr{T},\le)$ is strictly
quasi-pseudo-metrizable iff it is convex semiclosed preordered and
$(E,\mathscr{T}^\sharp,\mathscr{T}^\flat)$ is
quasi-pseudo-metrizable.

Every strictly quasi-pseudo-metrizable preordered space is a
quasi-pseudo-metrizable preordered space. Every
quasi-pseudo-metrizable preordered space is a completely regularly
preordered space \cite[Prop. 2.3]{minguzzi12b}.

Every $T_2$-ordered space is Hausdorff and ``every second countable
Hausdorff locally compact topological space is metrizable'' by
Urysohn's metrization theorem. The next result is a kind of order
generalization, which reduces to the just given statement for the
discrete order.

\begin{theorem} \label{ncu}
Let $(E,\mathscr{T}, \le)$ be a $T_2$-ordered space such that
$(E,\mathscr{T})$ is second countable and locally compact. If
$(E,\mathscr{T}, \le)$ is $k$-preserving or compactly generated,
then it is quasi-pseudo-metrizable. That is, there is a
quasi-pseudo-metric $p: E\times E\to [0,+\infty)$ (actually an
Albert's quasi-metric) such that $\mathscr{T}$ is the topology
induced by the metric $p\vee p^{-1}$ and
$G(\le)=\{(x,y):p(x,y)=0\}$.
\end{theorem}

\begin{proof}
Second countability implies the Lindel\"of property which under
local compactness is equivalent to $\sigma$-compactness. The
topological ordered space is a completely regularly ordered space
(quasi-uniformizable) by theorem \ref{bdd} (in the $k$-preserving
case) or theorem \ref{bdr}. Thus $E$ is a separable
quasi-pseudo-metric space by \cite[Theor.\ 2.5]{minguzzi12b}. The
pseudo-metric $p\vee p^{-1}$ is actually a metric by antisymmetry of
$\le$, thus $p$ is an Albert's quasi metric \cite{albert41}.
\end{proof}

We remark that we are not claiming that the topology induced by $p$
is the upper topology $\mathscr{T}^\sharp$ and that induced by
$p^{-1}$ is the lower topology $\mathscr{T}^\flat$ (which would be
true if we could prove strict quasi-pseudo-metrizability
\cite{minguzzi12b}).

\subsection{Strict quasi-pseudo-metrization from the $I$-space condition}

In order to prove the strict quasi-pseudo-metrizability of a
topological preordered space we assume the $I$-space condition.

Let us recall that a topological preordered space  is a {\em
regularly preordered space} if it is semiclosed preordered, (a) for
every closed decreasing set $A$ and closed increasing set $B$ of the
form $B=i(x)$ which are disjoint, $A\cap B=\emptyset$, it is
possible to find an open decreasing set $U$ and an open increasing
set $V$ which separate them, namely $A\subset U$, $B\subset V$, and
$U\cap V=\emptyset$, and (b) for every closed decreasing set $A$ of
the form $A=d(x)$ and closed increasing set $B$  which are disjoint,
$A\cap B=\emptyset$, it is possible to find an open decreasing set
$U$ and an open increasing set $V$ which separate them, namely
$A\subset U$, $B\subset V$, and $U\cap V=\emptyset$. A completely
regularly preordered space need not be regularly preordered
\cite[Example 1]{kunzi94b}. This is a crucial difference with
respect to the usual discrete-preorder version.

The problem of quasi-pseudo-metrization of a bitopological space was
considered  in Kelly's work \cite{kelly63} and has been extensively
studied over the years
\cite{salbany72,parrek80,romaguera83,romaguera84,raghavan88,romaguera90,andrikopoulos07,marin09}.
For bitopological spaces Kelly \cite[Theor.\ 2.8]{kelly63} obtained a
generalization of Urysohn's metrization theorem which in our
topological preordered space framework reads as follows

\begin{theorem} \label{nvv} (Kelly)
Let $(E,\mathscr{T},\le)$ be a convex regularly preordered space and
assume that both $\mathscr{T}^\sharp$ and $\mathscr{T}^\flat$ are
second countable, then $(E,\mathscr{T},\le)$ is strictly
quasi-pseudo-metrizable.
\end{theorem}


%
%
%
%

Under the $I$-space assumption it is possible to infer the second
countability of the coarser topologies $\mathscr{T}^\sharp$ and
$\mathscr{T}^\flat$ given that of $\mathscr{T}$, and hence we are
able to prove the next result.

\begin{theorem} \label{nci}
Every second countable locally convex locally compact
$T_2$-preordered $I$-space $(E,\mathscr{T},\le)$ is strictly
quasi-pseudo-metrizable (observe that local convexity holds whenever
the space is $k$-preserving or compactly generated, and the preorder
is antisymmetric, see Theor. \ref{bdd} and Cor. \ref{bdr}).
\end{theorem}

\begin{proof}
By theorem \ref{bfr} $E$ is convex. Let us prove that $E$ is a
regularly preordered space. Let $B$ be a closed increasing set and
let $x\in E\backslash B$. By Prop.\ \ref{pbf} strong local
compactness holds at $x$, thus there are an open set $O$ and a
compact set $K$, such that $x\in O\subset K\subset E\backslash B$.
The open decreasing set $d(O)$ is contained in the closed decreasing
set $d(K)$ which is disjoint from $B$. The proof in the dual case is
analogous, thus $E$ is regularly preordered. Let $\{O_i\}$ be a
countable base for $\mathscr{T}$, then $\{i(O_i)\}$ is a countable
base for $\mathscr{T}^\sharp$ and $\{d(O_i)\}$ is a countable base
for $\mathscr{T}^\flat$. By theorem \ref{nvv} $(E,\mathscr{T},\le)$
is strictly quasi-pseudo-metrizable .
\end{proof}

A  relevant application of this theorem is (see Sect.\ \ref{nod} for
definitions and basic results in causality theory)
\begin{theorem}
Globally hyperbolic, causally simple, and causally continuous
spacetimes endowed with the manifold topology $\mathscr{T}$ and the  order
$\overline{J^+}$ are strictly quasi-pseudo-metrizable topological
ordered spaces.
\end{theorem}

\begin{proof}
Under causal continuity $K^+=\overline{J^+}$ and $(M,\mathscr{T},
K^+)$ is compactly generated (Sect.\ \ref{jej}). Under causal
continuity the relation $\overline{J^+}$ sends open sets to open
sets (Remark \ref{mix}). Global hyperbolicity implies causal
simplicity which implies causal continuity (Sect.\ \ref{nod}).
\end{proof}
%
In other words, the strongest causality property met in causality
theory (i.e.\ global hyperbolicity) implies the strongest
preorder-separability condition.

\section{Conclusions}

In many applications the underlying mathematical structure involves
a topological  space $(E,\mathscr{T})$ endowed with a preorder
$\le$. If the preorder is not closed, it is usually convenient to
consider the smallest closed preorder containing it, and hence to
work in the framework of closed preordered spaces.

Quasi\hyph uniformizable topological preordered spaces are among the
most well behaved topological preordered spaces. They admit
completions and compactifications
\cite{nachbin65,fletcher82,liu97b,minguzzi11b}, and under second countability
they can be shown to be quasi-pseudo-metrizable \cite{minguzzi12b}.

In a previous work we established that every $T_2$-preordered
locally compact $\sigma$-compact space is normally preordered, and
hence that it is possible to obtain strong preorder-separability
properties imposing some topological conditions on $E$.
Unfortunately, normally preordered spaces are not necessarily
quasi-uniformizable, a fact that distinguishes the theory of
topological preordered spaces from the usual (discrete-preorder)
topology. In order to obtain the quasi\hyph uniformizability of the
topological preordered space it is necessary to prove its convexity.

This property is trivially satisfied in the discrete preorder case
and, as a consequence, results on the convexity of a topological
preordered space are particularly interesting as they have no analog
in the usual non-ordered topology.

We have proved that locally compact $\sigma$-compact locally convex
$T_2$-preordered spaces are convex, that is, imposing good
topological conditions on $E$  promotes local convexity to
convexity. This result is non-trivial because convexity is a global
property as it makes reference to  the openness of some monotone
sets over $E$.

Then we investigated conditions that guarantee local convexity under
antisymmetry. We proved that if the ordered space is such that the
convex hull of a compact set is compact ($k$-preserving) then
convexity holds. We also considered compactly generated preorders
proving that this condition together with the above topological
assumption on $E$, implies convexity.

In most applications the preorder is compactly generated (Sect.
\ref{jej}), thus we have indeed succeeded in proving the quasi\hyph
uniformizability of the corresponding topological preordered space,
and hence the possibility of compactifying it. For instance, a
spacetime is stably causal if and only if the relation $K^+$ of
example \ref{bws} is antisymmetric, in which case it coincides with
the Seifert's causal relation \cite{minguzzi08b,minguzzi09c}. From
our results a stably causal spacetime endowed with this relation is
quasi-uniformizable, Theor.\ \ref{nix} (and in fact
quasi-pseudo-metrizable). The Nachbin compactification allows us to
introduce a spacetime boundary and to extend the Seifert relation as
a closed relation on the whole compactified space.

The paper ends with some results on (strict)
quasi-pseudo-metrizability of second countable and locally compact
closed preordered spaces.

\section*{Acknowledgments}
I thank  H.-P.A. K{\"u}nzi for pointing out example \cite[Example
4.9]{fletcher82} and A. Fathi for pointing out Akin and Auslander's
works. This work has been partially supported by GNFM of INDAM.



\end{document}